\documentclass[final]{siamltex}

\usepackage{epsfig}
\usepackage{amsmath,amsfonts,amssymb}
\usepackage{algpseudocode,algorithm}
\newcommand{\smat}[1]{ \left[\begin{smallmatrix} #1 \end{smallmatrix}\right]}
\def\setC{\,\mathbb{C}}
\def\setR{\mathbb{R}}

\newtheorem{rumatemppu}{Example}
\newenvironment{example}{\begin{rumatemppu}\rm}{\end{rumatemppu}}

\title{GRADIENTS OF QUOTIENTS AND EIGENVALUE PROBLEMS} 
\author{
Marko Huhtanen\thanks{
Faculty of Information Technology and Electrical Engineering,
University of Oulu,
90570 Oulu 57,
Finland,
({\tt Marko.Huhtanen@aalto.fi}).}
\and
Olavi Nevanlinna \thanks{
Department of Mathematics and Systems Analysis,
Aalto University,
P.O. Box 1100,
FI-00076 Aalto,
Finland,
({\tt Olavi.Nevanlinna@aalto.fi}).
}
}
\begin{document}
\maketitle
\begin{abstract}
Intertwining analysis, algebra, numerical analysis and optimization, 
computing conjugate co-gradients of real-valued quotients 
gives rise to eigenvalue problems. In the linear Hermitian case,
by inspecting  optimal quotients in terms of taking the conjugate co-gradient for their critical points, a generalized folded spectrum eigenvalue problem arises.
Replacing the Euclidean norm in optimal quotients with the 
$p$-norm, a matrix version of the so-called $p$-Laplacian eigenvalue problem arises.
Such nonlinear eigenvalue problems seem to be naturally classified as being a special case of homogeneous problems. 
Being a quite general class, tools are developed for recovering whether a given homogeneous eigenvalue problem is a gradient eigenvalue problem. 
It turns out to be a delicate issue to come up with a valid quotient.
A notion of nonlinear Hermitian eigenvalue problem is suggested.
Cauchy-Schwarz quotients are introduced.
%Numerical methods based on availability of a quotient are derived. 
\end{abstract}


\pagestyle{myheadings}
\thispagestyle{plain}

\markboth{M. HUHTANEN AND O. NEVANLINNA
}{QUOTIENTS AND EIGENVALUE PROBLEMS}

\section{Introduction}
For the nonsingular\footnote{Nonsingular means there are nonsingular linear
combinations of the matrices $M$ and $N$.} eigenvalue problem
\begin{equation}\label{gen}
Mz=\lambda Nz,
\end{equation}
with matrices $M,N\in \setC^{n \times n}$, 
there exist the Rayleigh quotients 
\begin{equation}\label{rakyr}
  {\rm rq}_{M,N}(z)=\frac{(Mz,Nz)}{\|Nz\|^2}
\end{equation}
and the optimal quotients
\begin{equation}\label{bis1ar}
    {\rm oq}_{M,N}(z)= \frac{(Mz,Nz)}{|(Mz,Nz)|}
    \frac{\|Mz\|}{\|Nz\|}
  \end{equation}
to analyze the problem. (For the linear eigenvalue problem,
see, e.g., \cite{IKRA,PA} and the large number of references therein.) These can be analogously defined for operators.
Information they provide does not coincide in general, leading to different iterative methods \cite{HUKO} for numerically solving \eqref{gen}.
Still, both of these quotients suggest that 
the eigenvalue problem %\eqref{gen} 
deserves to be called Hermitian if
%\begin{equation}\label{herm} 
$N^*M$ is a Hermitian matrix;
%(Mz,Nz)\, \mbox{ is real for any } \,z \in \setC^n;
%\end{equation}
see \cite{HOOKOO,HNEAR}. Then, through computing conjugate co-gradients, a  
one-to-one correspondence between optimizing these quotients and solving eigenvalue problems arises. The Rayleigh quotients \eqref{rakyr} return the original eigenvalue problem whereas 
the optimal quotients \eqref{bis1ar} can be associated with the folded spectrum method as well as the generalized singular value decomposition. 
The folded spectrum method was introduced and is still used by physicists with Hermitian matrices \cite{WE,WZ}. 
Although not necessarily very fast, it is remarkable by allowing approximating
interior eigenvalues without applying the inverse.
In this paper first it is shown that the Hermitianity 
yields an extended class of eigenvalue problems admitting approximations with the folded spectrum method. 
Thereafter, again by computing conjugate co-gradients,  
the matrix version of the so-called $p$-Laplacian eigenvalue problem is shown to act as a natural gateway to nonlinear spectral theory. 

An important family of nonlinear eigenvalue problems arise from  taking the conjugate co-gradient to optimize a quotient 
\begin{equation}\label{ksr}
\frac{f(z,\overline{z})}{g(z,\overline{z})}
\end{equation}
involving two sufficiently smooth real valued functions $f:\setC^n\rightarrow \setR$ and 
$g:\setC^n\rightarrow \setR$. 
The matrix version of the $p$-Laplacian eigenvalue problem is obtained by taking the conjugate co-gradient
of the optimal quotients \eqref{bis1ar} 
once the Euclidean norm has been replaced with the $p$-norm; see \eqref{fek}. 
It appears that homogeneous eigenvalue problems 
\begin{equation}\label{puuro}
A(z,\overline{z})=\lambda B(z,\overline{z}),
\end{equation}
with two sufficiently smooth functions $A$ and $B$ on $\setC^n$,
provide an appropriate  setting to study such nonlinear
problems; see Definition
\ref{houmou}. 
Then a number of familiar attributes are preserved
although basic questions such as 
non-emptiness of the spectrum are non-trivial.  
Still, one should be aware that \eqref{puuro} encompasses
a wide variety of problems such as the $\setR$-linear eigenvalue problem and, curiously, even solving linear systems. 
To establish non-emptiness of the spectrum and produce eigenvalue approximations, it is natural to inspect whether a given homogeneous eigenvalue problem is a gradient eigenvalue problem. 
This is a two step process as follows. First one needs to construct
a candidate 
\begin{equation}\label{hurn}
\frac{{\rm Re}\,(A(z,\overline{z}),z)}{{\rm Re}\,(B(z,\overline{z}),z)}
\end{equation}
for the quotient; see Theorem \ref{candi}.
Thereafter one needs to check whether taking the conjugate co-gradient actually returns the eigenvalue problem. If it does, then it provides  a certain notion of non-linear Hermitianity, requiring non-vanishing of $g$ for
$z\not=0$; see
Corollary \ref{ser}. This is the optimal situation.
If these conditions are not met, using \eqref{hurn} in eigenvalue estimation can be highly questionable; see Example \ref{paha} in particular. For the matrix version of the 
$p$-Laplacian the approach is particularly successful, allowing approximating not just extreme but also other eigenvalues with quotients; 
see \eqref{skew}.

If \eqref{puuro} fails to be a gradient eigenvalue problem, then it is a precarious effort to produce eigenvalue estimates using \eqref{hurn}. There is one quotient that can be considered.
Measuring linear independence in terms of the so-called Cauchy-Schwarz quotient
yields then an arguable
general option to proceed for approximating eigenvectors. It also provides a measure of emptiness of the spectrum.
Algorithms based on these observations can be expected to be slow but at least they provide a way to proceed to produce approximations in tough problems where no other quotients are available. Then a natural candidate to produce eigenvalue estimates is to extend the optimal quotients \eqref{bis1ar} to nonlinear eigenvalue problems.

The paper is organized as follows. In Section \ref{root} 
taking the conjugate co-gradient in the (linear) Hermitian case is studied. It is explained how two dimensional matrix subspaces containing a positive definite element (which is an often encountered structure in numerical linear algebra) are in correspondence with Hermitian eigenvalue problems.
In Section \ref{nonline} conjugate co-gradients of quotients 
giving rise to nonlinear eigenvalue problems are inspected.
Homogeneous eigenvalue problems are formulated. A method to recover gradient eigenvalue problems is devised.
Several examples are given.
In Section \ref{kasva} the problem of finding a quotient 
for any nonlinear eigenvalue problem is addressed. 
The Cauchy-Schwarz quotient is devised to approximate eigenvectors. 
Optimal quotients are then proposed for estimating eigenvalues.

\section{Gradients of quotients and spectrum folding
for Hermitian eigenvalue problems}\label{root}
It is not exaggeration to say that in the early applications of matrix theory and physics\footnote{To quote from \cite[p.334]{KATO} dating from the 1940s, "...while non-self-adjoint operators are so rare in applications that we may always assume that the operators under consideration are self-adjoint, even when it is not explicitly proved."}, matrices and operators were Hermitian. 
The fact that quadratic forms and Hermitian matrices are in 
one-to-one correspondence led then to the usage of variational principles in eigenvalue estimation. 
These principles are still central. 
In view of extending this, 
consider optimizing the quotient \eqref{ksr}.
%\begin{equation}\label{ksr}
%\frac{f(z,\overline{z})}{g(z,\overline{z})}
%\end{equation}
%involving two sufficiently smooth real valued functions $f:\setC^n\rightarrow \setR$ and $g:\setC^n\rightarrow \setR$. 
A reason for working in the complex domain is, besides physics, to cover the standard Hermitian eigenvalue problem.
%(For real considerations, see \cite{CHI}.)
Moreover, in electrical engineering there are lot of applications concerned with optimizing 
real valued functions on $\setC^n$; see the highly cited paper \cite{KD} illustrating the multitude of applications involving such problems.
%Observe that if the eigenvalue problem \eqref{gen} is Hermitian, i.e., $N^*M$ is a Hermitian matrix, then both of the quotients \eqref{rakyr} and \eqref{bis1ar} satisfy this assumption.

Under these assumptions, taking the conjugate co-gradient 
with $$\frac{\partial}{\partial \overline{z}}=
(\frac{\partial}{\partial \overline{z}_1},\ldots, \frac{\partial}{\partial \overline{z}_n})$$ for the critical points  suffices. 
(Apparently, a first serious application of this in 
\cite{BRA}.)
This gives rise to  the equation
\begin{equation}\label{alkunabla}
\frac{1}{g(z,\overline{z})}\left(\frac{\partial}{\partial \overline{z}} f(z,\overline{z})-\frac{f(z,\overline{z})}{g(z,\overline{z})} 
\frac{\partial}{\partial \overline{z}} g(z,\overline{z})\right)=0
\end{equation}
or, equivalently, to 
\begin{equation}\label{nabla}
\frac{\partial}{\partial \overline{z}} f(z,\overline{z})-\frac{f(z,\overline{z})}{g(z,\overline{z})} 
\frac{\partial}{\partial \overline{z}} g(z,\overline{z})=0.
\end{equation} 
This implies, because the quotient is a scalar, 
that the conjugate co-gradients must be parallel at a critical point.
So replacing 
$\frac{f(z,\overline{z})}{g(z,\overline{z})}$ 
with the parameter $\lambda$ yields the so-called gradient eigenvalue problem associated with optimizing the quotient \eqref{ksr}. Clearly, we may assume 
$f(0,0)=g(0,0)=0$.
 
Observe that freeing 
$\frac{f(z,\overline{z})}{g(z,\overline{z})}$ means that
$\lambda$  can be complex. This may take place already in linear problems; see
Corollary \ref{ser}. This means that the equations can also be regarded as being  a complex version of the method
of Lagrange multipliers. Hence all the results apply to this central optimization problem as well.

\begin{definition} The eigenvalue problem \eqref{gen} is said to be Hermitian if $N^*M$ is a Hermitian matrix.
\end{definition}

For an algebraic interpretation, this means that the left eigenvectors of \eqref{gen} constitute an orthonormal basis \cite[Section 4]{HOOKOO}. If, on the other hand,  $M^*N$ is a Hermitian matrix, then the right eigenvectors constitute an orthonormal basis.

If the eigenvalue problem \eqref{gen} is Hermitian, then both the quotients 
\eqref{rakyr} and \eqref{bis1ar} satisfy the assumptions on
taking the conjugate co-gradient for the critical points.
In view of this, let us next inspect the linear 
eigenvalue problems resulting from taking the conjugate co-gradient
of the Rayleigh quotients and the optimal quotients.
Assuming Hermitianity, for the quotient
\eqref{rakyr} this construction returns the original eigenvalue problem \eqref{gen} 
in the form
\begin{equation}\label{ksw}
N^*Mz=\lambda N^*Nz,
\end{equation}
i.e., multiplied with $N^*$ from the left. 
As opposed to this, squaring the quotient \eqref{bis1ar}
gives 
\begin{equation}\label{ste}
 {\rm oq}_{M,N}(z)^2=\frac{\|Mz\|^2}{\|Nz\|^2}
\end{equation}
 which obviously has the same critical points.
The respective gradient eigenvalue problem then reads
\begin{equation}\label{eas}
M^*Mz=\lambda N^*Nz
\end{equation}  
instead.
This particular eigenvalue problem appears in connection with the so-called generalized singular value decomposition \cite{VALO}.
 We argue that this can also be interpreted as being the (generalized) folded spectrum eigenvalue problem in the case
the eigenvalue problem \eqref{gen} 
is Hermitian. 
The folded spectrum method, introduced by physicists \cite{WE,WZ} in the case $N=I$, is based on squaring the Hermitian matrix $M$ if eigenvalues are searched near the origin.
More generally, the method is meant for producing eigenvalue approximations 
without applying the inverse to  very large Hermitian matrices near a given point typically located in the interior of the spectrum. 
For $N\not=I$ the squaring of $M$ gets replaced with \eqref{eas}.  The aim is still to produce eigenvalue approximations without applying the inverse. This is a major challenge in general. 

\begin{theorem} Assume the eigenvalue problem \eqref{gen} is Hermitian. Then its eigenvalues are square roots of the eigenvalues of \eqref{eas}, with signs to be chosen,  possessing the respective eigenvectors.
\end{theorem}

\begin{proof}
If the eigenvalue problem \eqref{gen} is Hermitian, then
there exist a unitary matrix $U\in \setC^{n \times n}$ and an invertible matrix 
$X\in \setC^{n\times n}$ such that 
$$U^*MX =\Lambda_1\, \mbox{ and }\,U^*NX=\Lambda_2$$ 
with real diagonal matrices $\Lambda_1$ and $\Lambda_2$; see \cite{HOOKOO}. Consequently, \eqref{eas} can be written
as 
\begin{equation}\label{sqrt}
X^{-*}\Lambda_1^2X^{-1}=\lambda  X^{-*}\Lambda_2^2X^{-1}
\end{equation}
proving the claim. 
\end{proof}

If the problem is Hermitian, finding eigenvalues 
close to the origin means minimizing ${\rm oq}_{M,N}(z)^2$. 
For finding eigenvalues near to a given $\mu\in \setR$,
minimize ${\rm oq}_{M-\mu N,N}(z)^2$ 
with the corresponding gradient eigenvalue problem being
$$(M-\mu N)^*(M-\mu N)z=\lambda N^*Nz$$
instead.

Conversely, given an eigenvalue problem involving Hermitian matrices, there are instances one may define its root as being a Hermitian eigenvalue problem which returns it through forming \eqref{eas}. (Here "root" refers to how eigenvalues behave under the operation \eqref{eas}.)
Let us formulate this as follows. 

\begin{corollary} Let $A,B\in \setC^{n \times n}$ be Hermitian such that ${\rm span}_{\setR}\{A,B\}$ contains a positive definite element. Then 
${\rm span}_{\setR}\{A,B\}$
has a square root, i.e.,
there exists a Hermitian eigenvalue problem \eqref{gen} satisfying 
\begin{equation}\label{rootti}
{\rm span}_{\setR}\{M^*M,N^*N\}=
{\rm span}_{\setR}\{A,B\}
 \end{equation}
under folding.
\end{corollary}

\begin{proof} After possibly taking linear combinations, we may assume $A$ and $B$ to be positive definite.
Then for the eigenvalue problem
$$Az=\lambda Bz$$ there exists an invertible $X\in \setC^{n \times n}$ such that it can be written as \eqref{sqrt} for some  diagonal matrices
$\Lambda_1$ and $\Lambda_2$ with positive diagonal entries.
This relies on standard arguments; see, e.g., 
\cite[Theorem 15.3.2]{PA}. That is, take the Cholesky factorization $B=LL^*$ of $B$ and for $L^{-1}AL^{-*}$
its unitary diagonalization $L^{-1}AL^{-*}=Q\Lambda Q^*.$
Then set $X=Q^*L^*$ and $\Lambda_1^2=\Lambda$ and $\Lambda_2^2=I$. 

This allows to construct a root involving 
$M$ and $N$ giving a Hermitian problem such that the unitary matrix $U \in \setC^{n \times n}$ can be freely chosen. 
\end{proof}

There are a lot more Hermitian eigenvalue problems than
folded eigenvalue problems obtained by performing folding.
Quantitatively put, the dimension of Hermitian eigenvalue 
problems is about
$\frac{3n^2}{2}$ whereas the dimension of the folded 
eigenvalue problems is about $n^2$. The structure obtained under folding is well-known and typically
(to our mind misleadingly)
 called as "definite Hermitian generalized eigenvalue problem" \cite{PA}.
That is,  a two dimensional Hermitian matrix subspace containing positive definite elements.

Although the Rayleigh and optimal quotients lead to different eigenvalue problems through taking the conjugate co-gradient, in eigenvalue estimation they intertwine as follows.
Here $\sigma(M,N)$ denotes the spectrum of \eqref{gen}.

\begin{theorem} Assume the eigenvalue problem \eqref{gen} is Hermitian with $N$ invertible. Then for any $\mu \in \setC$ and nonzero
$z\in \setC^n$ holds
$${\rm dist}(\mu,\sigma(M,N))^2\leq
|{\rm rq}_{M,N}(z)-\mu|^2+ 
{\rm oq}_{M,N}(z)^2-{\rm rq}_{M,N}(z)^2.$$
\end{theorem}

\begin{proof} Since the eigenvalue problem \eqref{gen} is Hermitian, it follows that 
$MN^{-1}$ is a Hermitian matrix. Thereby
$${\rm dist}(\mu,\sigma(M,N))\|Nz\|
=\frac{\|Nz\|}{\|((MN^{-1}-\mu I)^{-1}\|}\leq
\|(M-\mu N)z\|.$$
Denote the unit vector $\frac{Nz}{\|Nz\|}$ by $q$ and set $A=MN^{-1}$. Then using twice the
 Pyhagorean theorem gives
$$\|(A-\mu I)q\|^2=\|(q^*Aq-\mu)q\|^2+\|(A-q^*Aq I)q\|^2=$$
$$|q^*Aq-\mu|^2+\|Aq\|^2-|q^*Aq|^2.$$
Now combining these two inequalities gives the claim.
\end{proof}

The least upper bound is given by the choice 
$\mu ={\rm rq}_{M,N}(z)$; see \cite[Section 4.5]{PA}
for the case $N=I$.
This is, however, quite deceptive since it does not mean the Rayleigh quotients give better approximations than the optimal quotients. (After all, we are dealing with an inequality.)
For example, if $N=I$ and one approximates the largest eigenvalue of $M$, then optimal quotients
always yield better approximations. This is probably the most important case since iterative methods, through many variants of the power method, rely exclusively on estimating largest eigenvalues; see \cite[Section 3.1]{HUKO}. 

For a Hermitian eigenvalue problem \eqref{gen}, consider producing approximations to eigenvalues near a given point $\mu\in \setR$. Bearing in mind the folded spectrum method, the aim is to take optimal steps while avoiding applying the inverse. 
%Having the option to avoid applying the inverse is, after all, a notable advantage in large eigenvalue computations.
To this end, the conjugate co-gradient
of the objective function \eqref{ste}  at $z$ is 
$$(M-\mu N)^*(M- \mu N)z-\frac{((M-\mu N)z,(M-\mu N)z)}{(Nz,Nz)}N^*Nz$$
determining the direction of the steepest descent (ascent).  Orthonormalize $z$ and the co-gradient to have
$z=z_1$ and $z_2$. Then consider the objective function
$$\frac{((\tilde{M}-\mu \tilde{N})v,
(\tilde{M}-\mu \tilde{N})v)}{\|\tilde{N}v\|^2}$$
with $\tilde{M}=MZ$ and $\tilde{N}=NZ$, where
$Z=[z_1\,z_2]$. Here $v\in \setC^2$.
For its critical points, solve the tiny $2$-by-$2$ eigenvalue problem 
\begin{equation}\label{mini}
(\tilde{M}-\mu \tilde{N})^*
(\tilde{M}-\mu \tilde{N})v=\lambda \tilde{N}^*\tilde{N}v.
\end{equation}
The eigenvector $v$ associated with the smallest (largest) eigenvalue is locally optimal once we set $Zv$. 
%This can be repeated by replacing $z$ with $Zv$ and, possibly,  thereafter $\mu$ respectively with \eqref{bis1ar}. 
At this point there are many ways to proceed. Algorithm \ref{alg:basic} is a basic version relying on a linear growth of storage. (Restarting this algorithm provides a method consuming a fixed amount of storage.) It yields monotonic convergence towards the eigenvalue nearest to $\mu$, assuming the starting vector is generic. This is because at the step 8  
we are minimizing ${\rm oq}_{M-\mu N,N}(z)$ with $z$ restricted to the subspace generated so far. This subspace gets enlarged at each step.

\begin{algorithm}[t]
   \caption{to approximate an eigenvalue $\lambda$ of 
   a Hermitian $Mz=\lambda Nz$ near $\mu$}
   \label{alg:basic} 
   \begin{algorithmic}[1]
     \State Read $n$-by-$n$ matrices $M$ and $N$ 
      %\While{ $\sigma_2([Mq\;Nq])>\epsilon$ }  
      \State Read an approximate eigenvector $z$ to approximate an eigenvalue near $\mu$
\State set $Z=[z]$      
      \For{$l=1,\ldots, j$}
          \State compute $w=(M-\mu N)^*(M- \mu N)z-\frac{((M-\mu N)z,(M-\mu N)z)}{(Nz,Nz)}N^*Nz$
          \State orthonormalize $w$ against the columns of 
          $Z$ to have $v$ and set $Z=[Z\,v]$
\State set  $\tilde{M}=MZ$ and $\tilde{N}=NZ$
\State  solve   $(\tilde{M}-\mu \tilde{N})^*
(\tilde{M}-\mu \tilde{N})v=\lambda \tilde{N}^*\tilde{N}v$          
 for the smallest eigenvalue and the corresponding eigenvector $v$              
 \State set $z=Zv$.
                     \EndFor
\State compute $\lambda =\frac{(Mz,Nz)}{|(Mz,Nz)|}
    \frac{\|Mz\|}{\|Nz\|}$                
 \end{algorithmic}
\end{algorithm}

%This can be regarded as a shift-without-invert method. 

\smallskip

\begin{example} To gain some immediate intuition into how Algorithm \ref{alg:basic} behaves,
assume $N=I$ and $\mu=0$ to start the iteration. Then $M$ is Hermitian, so that
we obtain the same subspace as with the Hermitian Lanczos method for $M^2$.  
\end{example}

\smallskip 

This means that it is unrealistic to expect the method to be very fast. After all, remarkable speeds (qubic or quadratic) attained with shift-and-invert techniques cannot be expected. So in realistic computations, to gain speed, one most likely wants to consider applying some type of inexact shift-and-inverse methods combined with optimization techniques; see \cite[Section 3.2]{HUKO}.

To end this section with two examples, bear in mind that in applications the inner-product is almost always problem dependent.
Then everything applies once the inner-product is set  accordingly. 

\smallskip

\begin{example}\label{femmija} It is not uncommon that a discretization of PDE using FEM leads to an eigenvalue problem \eqref{gen}
with $M$ and $N$ Hermitian such that
\begin{equation}\label{sar}
        \left[ \begin{array}{cc}
                M_{11}&M_{12}\\
                M_{12}^*&M_{22}
        \end{array} \right]\left[ \begin{array}{c}
        x_{1}\\
        x_{2}
        \end{array} \right]
        =\lambda
\left[ \begin{array}{cc}
        0&0\\
        0&N_{22}
\end{array} \right]\left[ \begin{array}{c}
x_{1}\\
x_{2}
\end{array} \right];
\end{equation}
see \cite{BEBO}. If $M_{11}$ is invertible
and $N_{22}$ positive definite,
then the problem is self-adjoint, i.e., the respective quotients \eqref{rakyr} and \eqref{bis1ar} are real-valued. 
Take $Z=\smat{I&0\\-M_{12}^*M_{11}^{-1}&I}$.
Now $$P=Z^*\left[ \begin{array}{cc}
I&0\\ 0&N_{22}^{-1}
\end{array} \right]
Z$$ yields a self-adjoining inner-product for \eqref{sar}.
That is, $(Mx,Nx)_P=(PMx,Mx)\in \setR$ for any $x \in \setC^n$.
\end{example}

\smallskip

\begin{example} Assume having a quadratic matrix pencil
$p(\lambda)=\lambda^2A_2+\lambda A_1 +A_0$ with 
$A_j\in \setC^{n \times n}$ for $j=0,1,2$.
This is equivalent with
$\lambda N-M$, where
$$M=\left[ \begin{array}{cc}
A_0&A_1\\ 0&I
\end{array} \right]\, \mbox{ and }
\,N=\left[ \begin{array}{cc}
0&-A_2\\ I&0
\end{array} \right].$$ This is obtained by setting
$A(\lambda)=\smat{p(\lambda)&0\\0&I}$ to have
$$M-\lambda N=V(\lambda)A(\lambda)W(\lambda)$$
with $V(\lambda)=\smat{I&\lambda A_2+A_1\\0&I}$ and
$W(\lambda)=\smat{I&0\\-\lambda I&I}$.
If $-A_0^*A_2$ is positive definite, then define
an inner product with $P=\smat{I&0\\0&-A_ 0^*A_2}$ in
$\setC^{2n}$. Then $(Mz,Nz)_P\in \setR$ for any $z\in \setC^{2n}$ if and only if $A_2^*A_1$ is a Hermitian matrix.
\end{example} 

\section{Gradients of quotients for homogeneous nonlinear eigenvalue problems}\label{nonline}
The fraction \eqref{ksr} can obviously be formulated  in such a way that the 
respective gradient eigenvalue problem \eqref{nabla}
is no longer linear. For an illustration, let us describe a well-know nonlinear PDE 
eigenvalue problem, i.e., the so-called nonlinear Rayleigh quotient problem
for the $p$-Laplacian \cite{GAPA,linkku, linkku2}. 
(In \cite{plap,linkku2} applications where
the problem appears have been listed.)
There one deals, in the corresponding matrix formulation, with minimizing 
the quotient 
\begin{equation}\label{fek}
\frac{\|Mz\|_p^p}{\|z\|_p^p}
\end{equation}
involving the $p$-norm. Here the matrix $M\in \setC^{n \times n}$ is given and $N=I$.\footnote{Everything here could  equally well be defined for a rectangular matrix $M$ and $N\not=I$.} 
It is apparent that naming this fraction  after "Rayleigh"  is notably misleading by the fact that the quotient \eqref{fek} is de facto the $p$-norm version of the optimal quotient squared in \eqref{ste}. Consequently, we are dealing with the respective "nonlinear" spectrum folding once the conjugate co-gradient is computed. 

To this end, since the objective function is real-valued,
% no partial derivation with respect to  $\frac{\partial}{\partial \overline{z}}$ is needed since the gradient is determined 
partial derivation with respect to $\frac{\partial}{\partial \overline{z}}$ for critical points suffices. Once the conjugate co-gradient 
$\frac{\partial}{\partial \overline{z}}$ is taken, we obtain
\begin{equation}\label{fors}
M^*(|Mz|^{p-2}\circ Mz)=\lambda |z|^{p-2} \circ z,
\end{equation}
where $\circ$ denotes the Hadamard product.
(It suffices to look at the $2$-by-$2$ case to see this structure. See also \cite[Proposition 1.2]{BEN} for an elegant dual vector formulation.)  All the operations on vectors involving absolute values
are taken entrywise.  In the nonlinear PDE eigenvalue problem $M$ is the gradient operator and $M^*$ the divergence operator.

It turns out to be instructive to consider this nonlinear eigenvalue problem slightly more generally by setting    
\begin{equation}\label{nors}
M_1(|M_2z|^{p-2}\circ M_3z)=\lambda |z|^{p-2}\circ z
\end{equation}
with matrices $M_j \in \setC^{n \times n}$ and $1<p<\infty$. 
This is a %nonlinear 
homogeneous %eigenvalue 
problem of the form
\begin{equation}\label{swes}
A(z,\overline{z})=\lambda B(z,\overline{z})
\end{equation}
with two sufficiently smooth functions $A$ and $B$ on $\setC^n$.

\begin{definition}\label{houmou} 
An eigenvalue problem  \eqref{swes}
is said to be homogeneous\footnote{A more accurate characterization would be $\setR$-homogeneous.} if there exist $k,l \in \setR^+$ 
such that for any non-zero $\alpha=re^{i \theta}\in \setC$ 
$$A(\alpha z, \overline{\alpha z})=
r^kA(e^{i \theta} z,\overline{e^{i \theta}z})\,
\mbox{ and }\, B(\alpha z, \overline{\alpha z})=r^l
B(e^{i \theta} z,\overline{e^{i \theta}z})$$
holds.
\end{definition} 

Due to the standard eigenvalue problem \eqref{gen}, the most often encountered  case consists of having $k=l=1$.
Generalizing the quotient \eqref{fek}, the simplest example
with $k\not=l$ appears
in \cite{GAPA}. There the authors consider the 
nonlinear Rayleigh quotient problem also in the case of 
having $q$-norm of $z$ in the denominator with $q\not=p$. 
In general, the homogeneity 
makes the problem  compact as follows, guaranteeing the closedness of the spectrum.  

\begin{theorem} 
Assume the eigenvalue problem \eqref{swes} is homogeneous 
such that   
$\{z\not=0\;| \; A(z,\overline{z})=B(z,\overline{z})=0 \}=\emptyset$. Then it suffices to inspect the eigenvalue problem on the unit sphere of $\setC^n$. 
Moreover, the set of eigenvalues  is closed. It is compact if $k=l$.
\end{theorem}

\begin{proof} 
Suppose $z$ is an eigenvector corresponding to 
an eigenvalue $\lambda$. Then $rz$ with $r\in \setR$ is an eigenvector corresponding to $r^{l-k}$, i.e., the eigenvalues
scale by this monomial factor. It hence suffices to study the eigenvalue problem on the unit sphere of $\setC^n$.
In particular, eigenvectors are always non-zero, like in the linear case.

Let the eigenvectors be of unit length.
Then suppose there are eigenvalues $\lambda_j$ converging towards a finite $\lambda$.
Since the unit sphere of $\setC^n$ is compact, the corresponding eigenvectors contain a converging subsequence. 
Let $z$ be that limit. Then by continuity 
$A(z,\overline{z})=\lambda B(z,\overline{z})$.
If $\lambda$ is the infinity, then $B(z,\overline{z})=0$ and, by assumption, $A(z,\overline{z})\not= 0$.
Suppose $z$ is an eigenvector corresponding to 
an eigenvalue $\lambda$.  
If $k=l$, then the eigenvalues do not scale if the eigenvectors are scaled and, thereby, the spectrum is closed and bounded (plus possibly $\infty$).
\end{proof}

The compactness of the spectrum in the case $k=l$ means that the respective quotient attains its minimum and maximum, implying the following.

\begin{corollary}
Assume $k=l$ and the problem is a gradient eigenvalue problem such that $g$ does not vanish for $z\not=0$.
Then the spectrum is non-empty.
\end{corollary}

\begin{proof} The image of the unit ball of $\setC^n$ is compact. Thereby the maximum and minimum of
\eqref{ksr} is attained. To these correspond critical points
since homogeneity guarantees that there is no growth or decay in the radial directions.
\end{proof}

Assume having a homogeneous eigenvalue problem with 
$k\not=l.$ Let us say $k>l$. To compactify the spectrum without any loss of relevant information, consider replacing  
$B(z,\overline{z})$ with 
$\|z\|^{k-l}B(z,\overline{z}).$  
This transformation leads to the homogeneous eigenvalue problem 
\begin{equation}\label{swesuni}
A(z,\overline{z})=\lambda \|z\|^{k-l}B(z,\overline{z})
\end{equation}
which retains the eigenvectors and eigenvalues for unit eigenvectors of the original problem. This formulation has a compact spectrum. It provides a lot flexibility,
even covering solving linear systems as follows. 

\smallskip

\begin{example}\label{lineq} 
Let $M\in \setC^{n \times n}$ be invertible
and consider solving a linear system
\begin{equation} \label{solu}
Mz=b
\end{equation}
for $b\in \setC^n.$ This can be viewed as an eigenvalue problem with $B(z,\overline{z})$ being 
the  simplest possible function, i.e., 
constant and hence independent of $z$, once we set 
\begin{equation}\label{cohen}
Mz=\lambda b.
\end{equation} 
Then $k=1$ and $l=0$.
Now performing the homogenization \eqref{swesuni} gives
$$Mz=\lambda \|z\| b,$$
so that solving a linear system is converted into 
an equivalent problem of solving a homogeneous eigenvalue problem with $k=l=1$. The spectrum consists of the circle of radius $\frac{1}{\|M^{-1}b\|}$ centred at the origin.
The eigenvectors are of the form  $\alpha z$ with a nonzero 
$\alpha\in \setC$ and
$z$ being the solution to \eqref{solu}.
\end{example}

\smallskip

These familiar attributes may give a false sense of security.
That is, although closed, the spectrum can be empty. 
It can also contain a continuum, like the preceding example
illustrates. These two extremes 
cannot be ignored, in particular, since
they take place already in connection with the 
$\setR$-linear eigenvalue problem \cite{HUPFA,HUNE}.
The $\setR$-linear eigenvalue problem can be regarded as being  the most natural extension of the classical $\setC$-linear eigenvalue problem because of the $\setR$-linearity and the fact that $k=l=1$. It provides a rich source of instructive and spectrally 
staggeringly varied homogeneous eigenvalue problems.
In any event, establishing non-emptiness of the spectrum in the homogeneous case is a fundamental task. Let us concentrate on recovering whether we are dealing with a
gradient eigenvalue problem. 

\begin{theorem}\label{candi}
 Assume a homogeneous eigenvalue problem \eqref{swes} is a 
gradient eigenvalue problem \eqref{nabla}.
Then 
\begin{equation}\label{sarse}
\frac{{\rm Re}\,(A(z,\overline{z}),z)}{{\rm Re}\,(B(z,\overline{z}),z)}=\frac{(k+1)f(z,\overline{z})}{
(l+1)g(z,\overline{z})}
\end{equation}
assuming $f(0,0)=g(0,0)=0$. 
%the spectrum is a subset
%of $\setR$.
\end{theorem} 
 
\begin{proof} Set $F(r)=f(rz,r\overline{z})$ with 
$r\in \setR$. Differentiating gives 
$$\frac{dF}{dr}(r)=\sum_{j=1}^nz_j
\frac{\partial f}{\partial z_j}+
\sum_{j=1}^n\overline{z}_j
\frac{\partial f}{\partial \overline{z}_j}
=2r^k\sum_{j=1}^n
{\rm Re}\,(\overline{z}_j A(z,\overline{z}))=2r^k
{\rm Re}\,(A(z,\overline{z}),z)$$
by applying the chain rule. If the constant term of $F$ vanishes, i.e., $f(0,0)=0$, this yields after integrating both sides with respect
to $r$ that 
$$(k+1)f(z, \overline{z})=2{\rm Re}\,(
\frac{\partial}{\partial \overline{z}} f(z,\overline{z}),z)$$
once we let $r\rightarrow 1$.
%$f$ is homogeneous of degree $l+1$.
The same arguments can be used with $g$ to have
\begin{equation}\label{zero}
(l+1)g(z, \overline{z})=2{\rm Re}\,(
\frac{\partial}{\partial \overline{z}} g(z,\overline{z}),z)
\end{equation}
 so that diving these
yields the claim.
%Apply $z^*$ to both sides of \eqref{swes} from the left to
%cancel the partial derivation by the Euler's homogeneous
%function theorem. Divide both sides by 
%$(z,B(z, \overline{z}))$ to establish the reality of $\lambda$.
\end{proof} 

The preceding proof is essentially a complex version of Euler's homogeneous function theorem. 
Taking the real part is necessary. 
%This can be seen by inspecting, for example, $f(z,\overline{z})=({\rm Re}\,z)^2$.
Moreover, the "averaging" factor $\frac{k+1}{l+1}$ results from the differentiation in obtaining \eqref{swes}.
Let us illustrate these with the following example.

\smallskip

\begin{example} Continuing Example \ref{lineq},
consider \eqref{cohen} with
$M$ being Hermitian. Use \eqref{sarse} to have
$$\frac{{\rm Re}\,(A(z,\overline{z}),z)}{{\rm Re}\,(B(z,\overline{z}),z)}=
\frac{(Mz,z)}{\frac{1}{2}((b,z)+(z,b))}$$
which is a valid quotient. That is, it returns \eqref{cohen} with $M$ multiplied by two once the conjugate co-gradient is taken. Consequently, solving linear systems involving Hermitian matrices
can be formulated as homogeneous gradient eigenvalue problems.
\end{example}

\smallskip

The following extends what takes place in the standard Hermitian eigenvalue problem.

\begin{corollary}\label{ser}  
Assume a homogeneous eigenvalue problem \eqref{swes} is a 
gradient eigenvalue problem \eqref{nabla} such that
$g$ does not vanish for $z\not=0$. If 
$$(A(z,\overline{z}),z)\in\setR\, \mbox{ and }\, 
(B(z,\overline{z}),z)\in\setR$$ for every $z\in\setC^n$, then the spectrum is real and 
$$\left\{ \frac{(A(z,\overline{z}),z)}{(B(z,\overline{z}),z)}
\,:\, z\not=0\right\}$$
contains the eigenvalues.
Moreover, if $g$ does vanish for some $z\not=0$, then there can occur complex eigenvalues.
\end{corollary}

\begin{proof} Suppose $z\in \setC^n$ is an eigenvector. 
If  $z$ is not orthogonal against $B(z,\overline{z})$, then 
taking the inner-product of both sizes of \eqref{swes} with $z$ gives
$\frac{(A(z,\overline{z}),z)}{(B(z,\overline{z}),z)}$ which is the eigenvalue. This eigenvalue is thus real. If $z$ were orthogonal against
$B(z,\overline{z})$, then $z$ would be in the zero set of $g$
by \eqref{zero}. But this is a contradiction and thereby $z$ cannot be orthogonal against $(Bz,\overline{z})$. 

It suffices to consider the linear case to see that complex eigenvalues can occur if $g$ vanishes for some non-zero $z$.
That is, take the quotient \eqref{ksr} now to be 
\begin{equation}\label{juupa}
\frac{(Mz,z)}{(Nz,z)}
\end{equation}
with Hermitian matrices $M$ and $N$. Assume $N$ is invertible
such that $MN^{-1}$ has complex eigenvalues. (It is well-know that this is possible.) 
This forces $N$ to be indefinite, so that $(Nz,z)=0$ for some $z\not =0$.
\end{proof}

Thus, starting with a gradient eigenvalue problem, 
\eqref{sarse} can be used to locate real eigenvalues only.
It is noteworthy that \eqref{juupa} can be used to illustrate that there need not exist any real eigenvalues. 
In particular, these are not artificial constructions.
There are relevant applications involving linear gradient eigenvalue problems with complex eigenvalues \cite{PASYM}.
Then \eqref{rakyr} and \eqref{bis1ar} are valid quotients which should be used instead.

To establish non-emptiness, consider inspecting whether a given homogeneous eigenvalue problem \eqref{swes} is a gradient eigenvalue problem \eqref{nabla}. 
Recall that $f$ and $g$ are assumed to be real-valued, so that applying $\frac{\partial}{\partial \overline{z}}$
suffices for the critical points. Then, entirely due to homogeneity of the numerator and the denominator, applying $z^*$ is aimed at cancelling this partial derivation. Hence, to recover a quotient,
form a trial according to \eqref{sarse} by forming
\begin{equation}\label{skers}
\frac{{\rm Re}\,(A(z,\overline{z}),z)}{{\rm Re}\,(B(z,\overline{z}),z)}
\end{equation}
%returns $\frac{f(z,\overline{z})}{g(z,\overline{z})}$
%if the problem is a homogeneous gradient eigenvalue problem. 
%This can be used conversely as well.
%That is, starting with an eigenvalue problem \eqref{swes} one %obtains  
to have a candidate for a quotient to study.\footnote{There seems to be a tendency to call fractions of this type as "Rayleigh quotients". That is somewhat misleading. Rather, \eqref{skers} should 
%$\frac{\partial}{\partial \overline{z}}$
be interpreted as resulting from applying Euler's homogeneous function theorem.} Observe that this
fraction 
is invariant under multiplications of the eigenvalue problem
\eqref{swes} by a real-valued function. In particular, it makes no difference whether the formulation under consideration is
\eqref{alkunabla} or \eqref{nabla}.
 
%Here a nonsingular $Z\in \setC^{n \times n}$ should be chosen %so that
%the numerator and the denominator are real-valued functions.
%This operation can be interpreted as changing the geometry of %$\setC^n$ through preconditioning of \eqref{swes} 
%from the left; see \cite{HNEAR}, where it is shown to %generate a lot of flexibility to deal with $\setC$-linear %problems with real spectrum. (This flexibility is %indispensable, e.g., in non-Hermitian quantum mechanics.)

\smallskip

\begin{example}
Non-homogeneous eigenvalue problems can also be formulated although they resemble very little classical linear eigenvalue problems. First, they can be expected to fail the construction \eqref{sarse}.
For the simplest illustration, consider the discretized  Gross-Pitaevskii \cite{JKM} 
eigenvalue problem 
\begin{equation}\label{grop}
%A(z,\overline{z})=
Mz+ \beta |z|^2\circ z
=\lambda z
%z\circ z \circ \overline{z}
\end{equation}
 with a Hermitian matrix $M\in \setC^{n \times n}$ and $\beta \geq 0$. 
%We have $$(A(z,\overline{z}),B(z,\overline{z}))=(Mz,z)+\beta %(|z|^2,|z|^2)$$
%\sum_{j=1}^n|z_j|^4$$ 
%which is real for any $z\in\setC^n$. Therefore the problem
%can have only real eigenvalues. 
This is a gradient eigenvalue
problem \eqref{nabla} involving the quotient 
$$\frac{f(z,\overline{z})}{g(z,\overline{z})}
=\frac{(Mz,z)+\frac{\beta}{2} 
(|z|^2,|z|^2)}{(z,z)}$$
to optimize. This is not recovered with \eqref{sarse}.
Moreover, non-homogeneity means that the eigenvectors are not determined on the unit sphere. To illustrate what can go wrong, take a unit eigenvector $z_1$ of $M$, i.e., $Mz_1=\lambda_1z_1$.
Scale $z_1$ as $z=rz_1$ with $r>0$. Then
$A(z,\overline{z})=\lambda_1 z +r^3 z_1\circ z_1\circ \overline{z_1},$
so that the second term is $O(r^3)$.
In other words,
$$\lim_{r \rightarrow 0} \frac{A(z,\overline{z})-\lambda_1 z}{r}=0$$ 
as $A(z,\overline{z})$ points more and more into
the direction of $z$ as $r\rightarrow 0$.
\end{example}

\smallskip

Let us return to the matrix version of the nonlinear Rayleigh
quotient problem. It is a  nonlinear prototype where the construction described is successful.

\begin{theorem} For the eigenvalue problem \eqref{fors}
we have 
$$\frac{{\rm Re}\,(A(z,\overline{z}),z)}{{\rm Re}\,(B(z,\overline{z}),z)}=
\frac{\|Mz\|_p^p}{\|z\|_p^p}.$$
The eigenvalues are among these quotients such that the largest (smallest) eigenvalue equals the maximum (minimum)
of these quotients. 
\end{theorem}

\begin{proof} We have $(A(z,\overline{z}),z)=\|Mz\|_p^p$
and $(B(z,\overline{z}),z)=\|z\|_p^p$. Then apply Corollary \ref{ser} to conclude that all the eigenvalues are real. Moreover, the gradient vanishes at the eigenvalues and then the eigenvalue equals the quotient.
%The first claim follows by a direct computation.
%
%By \eqref{nabla}, 
%the critical points of \eqref{fek} correspond to eigenvalues
%of \eqref{fors} of the form \eqref{fek}. For the converse, if we have
%an eigenvector $z$ of \eqref{fors}, then
%apply $z^*$ from the left of the equality \eqref{fors}
%to recover $\lambda=\frac{\|Mz\|_p^p}{\|z\|_p^p}$. 
\end{proof}

Since the largest eigenvalue of \eqref{fors} is given by the $p$th power of the $p$-norm of the matrix $M$, it can be numerically approximated \cite{HI}. For the smallest, take the inverse of $M$ and proceed analogously.
It is also possible to device an ascent$/$descent method to approximate these eigenvalues by using the gradient \eqref{nabla}.
For possible other eigenvalues the approach is certainly less clear.
There is one particularly natural interpretation of the nonlinear eigenvalue problem \eqref{fors}. 

Namely, in the case $p=2$ we are solving for the singular values of the matrix $M$ by the fact that 
\eqref{fek} 
ignores the argument data which \eqref{bis1ar} contains. 
Thus the original nonlinear PDE eigenvalue problem 
involving the $p$-Laplacian
can also be viewed as a singular value problem for the gradient operator  studied using the $p$-norm. 
This is underscored by the fact that $M$ could equally well be a rectangular matrix and not necessarily square. %Consequently, it is natural to ask whether the $s$-numbers 
%\begin{equation}\label{asnu}
%\min_{{\rm rank}\,F<k}\|M-F \|_p =
%\min_{{\rm rank}\,F<k}
%\max_{z \in \setC^n}\frac{\|(M-F)z\|_p}{\|z\|_p}
%\end{equation}
%are related with the eigenvalues
%of \eqref{fors} like in the case $p=2$ through squaring. For unbounded operators, start from the bottom and consider
%\begin{equation}\label{inste}
%\max_{{\rm rank}\,F<k} \min \frac{\|(M-F)z\|_p}{\|z\|_p}
%\end{equation}
%with $z$ in the domain of definition of $M$. 
%(For the $s$-numbers in Banach spaces, see \cite{PI}.)
%It remains as an intriguing problem whether (or when) there exists a connection between these numbers and eigenvalues.
See \cite{BEN} where the problem has been elegantly solved in the positive in an operator theoretic problem.
Regarding our matrix problem, let us describe an approximation scheme. Take an eigenvector $u$ of \eqref{fors} corresponding to the largest (or smallest) eigenvalue. Its dual vector is $j(u)=\|u\|^{1-p}|u|^{p-2}\circ u$.
%Let $$M_+=M-<\cdot,j(u)>$$
%Any vector $z\in\setC^n$ can be decomposed as 
%$$z=\alpha u+\beta v$$ with $<v,j(u)>=0$ for some $\alpha,\beta \in \setC$.
%Then
%we have $M_+z=\beta Mu$.
Let $W\in \setC^{n \times (n-1)}$ be such that its columns
span the kernel of
$$z \longmapsto (z,j(u)).$$
Consider 
\begin{equation}\label{skew}
\max_{w\in \setC^{n-1}} \frac{\|MWw\|_p^p}{\|Ww\|_p^p}.
\end{equation}
Taking the conjugate co-gradient gives, after taking $W^*$ as a common factor, the condition
$$
W^*\left(M^*(|MWw|^{p-2}\circ MWw)- \lambda |Ww|^{p-2}\circ Ww\right)=0
$$
%$$\max_{w\in \setC^{n-1}} \frac{\|MWw\|_p^p}{\|Ww\|_p^p}$$
which is satisfied at the maximum (and minimum) point of \eqref{skew}.
Hence either 
\begin{equation}\label{genesis}
M^*(|MWw|^{p-2}\circ MWw)= \lambda |Ww|^{p-2}\circ Ww
\end{equation}
for some non-zero $Ww$, i.e., we have an eigenvalue of the
original eigenvalue problem \eqref{fors}. Or there exists a linear combination
of $M^*(|MWw|^{p-2}\circ MWw)$ and 
$|Ww|^{p-2}\circ Ww$ which is parallel with $j(u)$, the dual vector of the eigenvector corresponding to the largest eigenvalue. Consequently,  either new or repeated spectral information is obtained. This is also computationally a feasible approach. 
%(Again here methods suggested in \cite{HI} can be invoked with the matrix $MW$.)

It is noteworthy that quotients of the form  \eqref{skew} appear in the generalized singular value decomposition with $p=2$;
see \cite{VALO}.
%In the former case the steps can be repeated.

Aside from non-homogeneous eigenvalue problems, 
let us return to the eigenvalue problem \eqref{nors}
to see how \eqref{skers} can still fail to produce a correct quotient. 
As a first observation, this has
$\setC$-linear eigenvectors, i.e., if $z\in \setC^n$ is an eigenvector, then so is $\alpha z$ for any nonzero $\alpha \in \setC$.
In particular, with $M_1=M_3$ Hermitian and $p=2$ the eigenvalues are real and the eigenvectors
form an orthonormal basis. Observe that multiplying \eqref{nors} 
from the left with $|z|^{2-p}$ gives rise to
$$z \longmapsto {\mathcal L}(z)=|z|^{2-p}\circ M_1(|M_2z|^{p-2}\circ M_3z)$$
which is (in its domain of definition) a $\setC$-multiplicative operator. That is, 
${\mathcal L}(\alpha z)=\alpha \mathcal{L}(z)$. 
%\end{example}
%
%\smallskip
%In the PDE formulation $M_1$ is the divergence while $M_2$ and $M_3$ equal the gradient operator; 
%In \cite{plap,linkku2} applications where
%the problem appears have been listed.
%A matrix formulation of this is covered by the following more general setting. 

\begin{theorem} Assume $M_3=M_1^*=M$. Then
all the eigenvalues of \eqref{nors} are bounded real and non-negative. Moreover, in this case the trial \eqref{skers} yields a quotient giving the eigenvalue problem
\begin{equation}\label{uusi}
\frac{p-2}{2}M_2^*((|M_2z|^{p-4}|Mz|^2 \circ M_2z)+
M^*(|M_2z|^{p-2} \circ Mz)=\lambda |z|^{p-2}\circ z
\end{equation}
after taking the conjugate co-gradient.
\end{theorem}

\begin{proof} First we have 
$$|M_2z|^{p-2}\circ M_3z=D_{|M_2z|^{p-2}}M_3z,$$
where $D_{|M_2z|^{p-2}}$ is the diagonal matrix
having the diagonal entries given by the vector $|M_2z|^{p-2}$. Denote $M_3$ by $M$.
Thus \eqref{nors} can be written as
\begin{equation}\label{norsr}
(M^*D_{|M_2z|^{p-2}}M-\lambda D_{|z|^{p-2}})z=0.
\end{equation}
For this to have a nonzero solution, a necessary
condition is that the matrix 
\begin{equation}\label{inklu}
M^*D_{|M_2z|^{p-2}}M-\lambda D_{|z|^{p-2}}
\end{equation}
is singular.
Equip now $\setC^{n}$ with the inner-product
involving the positive definite matrix 
$P=D_{|z|^{p-2}}^{-1}$ 
 whenever all the entries of $z$ are nonzero. 
In this inner-product 
$(PA(z),B(z))=(M^*D_{|M_2z|^{p-2}}Mz,z)$ is real and non-negative. If $z$ happens to be an eigenvector, then also the
respective eigenvalue is real and non-negative.
%Hence the respective  field of optimal quotients is real and non-negative. 
When there are zero entries among $z$, then
orthogonally project the eigenvalue problem onto the subspace of 
$\setC^n$ with vanishing entries correspondingly. 
Whenever such $z$ is an eigenvector of the original problem,
it  is also an eigenvector of the projected problem.
Since the structure is preserved in the projected problem, proceed analogously to have the claim for the realness and non-negativity of the spectrum.

Form $(Mz)^*(|M_2z|^{p-2}\circ Mz)$ which is real-valued and thereby \eqref{skers} has real-valued denominator and numerator.  Compute the conjugate co-gradient to have the eigenvalue problem claimed.
\end{proof}

This theorem shows that it is possible to have
\begin{equation}\label{ssrs}
(A(z, \overline{z}),z)=(\tilde{A}(z, \overline{z}),z)
\end{equation}
for every $z \in \setC^n$ while $A\not=\tilde{A}$.
This cannot happen with matrices, i.e., the field of values
consists of zero only for the zero matrix. 
%However, already in the $\setR$-linear eigenvalue problem this is possible as the following proposition and example illustrate. 
In particular, if $T \in \setC^{n \times n}$ is skew-symmetric, then $A(z,\overline{z})=\| z\|^{k-1} T\overline{z}$ is homogeneous of degree $k$ such that 
$$(A(z,\overline{z}),z)=0$$ 
for every $z \in \setC^n$. In the $\setR$-linear case this means the following.

\smallskip

\begin{example}\label{paha} 
Consider the $\setR$-linear eigenvalue problem
$$Mz+T\overline{z}=\lambda z$$
with a Hermitian $M$ and skew-symmetric $T$. 
If $\|T^{-1}\| \|M\|< 1$, then the spectrum is empty;
see \cite[Corollary 3.12]{HUPFA}. However, now
\eqref{skers} equals $$\frac{(Mz,z)}{(z,z)}$$
for which the respective gradient eigenvalue problem
is $$Mz=\lambda z,$$
i.e., the spectrum consists of the eigenvalues of $M$.
\end{example}
 
\smallskip

Consequently, after forming a trial quotient \eqref{skers} 
it is necessary to compute its conjugate co-gradient to verify whether it actually returns the eigenvalue problem it was derived from. If it does not, the previous example shows that the quotient need not give any useful information regarding the eigenvalues of the original problem.
The following illustrates, already in the two dimensional case, the challenges of showing 
non-emptiness of the spectrum without having a quotient.

\begin{corollary} Assume $n=2$. Then the spectrum is non-empty.
\end{corollary}

\begin{proof} 
If either $M_2$ or $M_3$ is singular, then any 
$z\in \setC^n$ in the respective nullspace  
is an eigenvector of \eqref{nors} 
corresponding to the eigenvalue $\lambda=0$.

Assume next that $M_2$ and $M_3$ are nonsingular. Let us put $z=\smat{1\\w}\in \setC^2$
with $w \in \setC$. The equations are then
$$\left[ \begin{array}{cc}
|m_{11}|^2d_1 +|m_{12}|^2d_2&m_{11}\overline{m_{21}}d_1+
m_{12}\overline{m_{22}}d_2\\
m_{21}\overline{m_{11}}d_1 +m_{22}\overline{m_{12}}d_2 &
|m_{21}|^2d_1 +|m_{22}|^2d_2
\end{array} \right]
\left[ \begin{array}{c}
1\\
w\end{array} \right]=\lambda \left[ \begin{array}{c}
1\\
|w|^{p-2}w\end{array} \right],
$$
where $d_1=d_1(w)$ and $d_2=d_2(w)$ are the diagonal entries of
$D_{|M_2z|^{p-2}}$. Dividing these equations cancels $\lambda$ and gives
$$(m_{11}\overline{m_{21}}d_1+
m_{12}\overline{m_{22}}d_2)|w|^{p-2}w^2+
((|m_{11}|^2d_1 +|m_{12}|^2d_2)|w|^{p-2}
-|m_{21}|^2d_1 -|m_{22}|^2d_2)w
$$ 
\begin{equation}\label{lksr}
-m_{21}\overline{m_{11}}d_1-
m_{22}\overline{m_{12}}d_2=0
\end{equation}
which $w$ should satisfy.  Observe that we may take any unitary diagonal matrices $D_1$ and $D_2$ converting
\eqref{norsr} into 
$$D_1^*(D_1MD_2D_{|M_2z|^{p-2}}D_2^*M^*D_1^*-\lambda D_{|z|^{p-2}})D_1 z=0.$$
When $n=2$ let us choose $D_1={\rm diag}(1,\alpha)$ and $D_2=I$. Then \eqref{lksr} reads
$$ \overline{\alpha}(m_{11}\overline{m_{21}}d_1+
m_{12}\overline{m_{22}}d_2)|w|^{p-2}(\alpha w)^2+
((|m_{11}|^2d_1 +|m_{12}|^2d_2)|w|^{p-2}-|m_{21}|^2d_1 -|m_{22}|^2d_2)
\alpha w
$$ 
\begin{equation}\label{lksra}
-\overline{\alpha}(m_{21}\overline{m_{11}}d_1+
m_{22}\overline{m_{12}}d_2)=0, 
\end{equation}
i.e.,
$$ \alpha\left((m_{11}\overline{m_{21}}d_1+
m_{12}\overline{m_{22}}d_2) |w|^{p-2}w^2+
((|m_{11}|^2d_1 +|m_{12}|^2d_2)|w|^{p-2}-|m_{21}|^2d_1 -|m_{22}|^2d_2)
w\right)
$$ 
\begin{equation}\label{lksrsa}
-\overline{\alpha}(m_{21}\overline{m_{11}}d_1+
m_{22}\overline{m_{12}}d_2)=0.
\end{equation}
For any $z=\smat{1\\w}$, choose $\alpha$ in such a way that
\begin{equation}\label{ekak}
\alpha \left((m_{11}\overline{m_{21}}d_1+
m_{12}\overline{m_{22}}d_2)|w|^{p-2}w^2+
((|m_{11}|^2d_1 +|m_{12}|^2d_2)|w|^{p-2}-|m_{21}|^2d_1 -|m_{22}|^2d_2)
w\right)
\end{equation}
and 
\begin{equation}\label{tokak}
-\overline{\alpha}(m_{21}\overline{m_{11}}d_1+
m_{22}\overline{m_{12}}d_2)
\end{equation}
point in opposite directions in the complex plane.
For $w$ large in modulus, \eqref{ekak} dominates.
For $w$ small in modulus, \eqref{tokak} dominates. 
Because of continuity, these complex numbers must
cancel for some $w$. 
\end{proof}

Although the elaborateness of this proof is somewhat discouraging, next it is shown that there always exists a quotient to optimize for approximating eigenvectors (not eigenvalues) or assessing emptiness of the spectrum.

\section{Gradients of Cauchy-Schwarz quotients and eigenvalue estimates}\label{kasva}
%\begin{example}
%In a nonlinear Schr\"odinger equation \cite{TVZ}, after a discretization, one obtains  $A(z,\overline{z})=
%Mz+(a|z|^{p_1}+b|z|^{p_2})\circ z$ with real 
%constants $a$ and $b$ and real positive exponents $p_1$
%and $p_2$. Moreover, $B(z,\overline{z})=z$. 
%\end{example}
Consider having an eigenvalue problem \eqref{swes} which  is not necessarily neither a homogeneous nor a gradient eigenvalue problem.  
Since eigenvalue problems are concerned with linear dependency, it is possible to devise a quotient 
\eqref{ksr}
to maximize to recover eigenvectors.  This also applies to the method of Lagrange multipliers to approximate multipliers. 
The quotient suggested can be inspected without any assumptions on the spectrum
such as non-emptiness.  In particular, it also reveals the emptiness of the spectrum.
That is, due to the Cauchy-Schwarz inequality, a necessary and sufficient condition on having an eigenvector of \eqref{swes} can be based on the following notion.

\begin{definition}
For the eigenvalue problem \eqref{swes} 
 the Cauchy-Schwarz quotient is defined as 
\begin{equation}\label{reali}
%\frac{f(z,\overline{z})}{g(z,\overline{z})}=
\frac{|(A(z,\overline{z}),B(z,\overline{z}))|^2}{
\|A(z,\overline{z})\|^2 \|B(z,\overline{z})\|^2}
\end{equation}
for $z\in \setC^n$. 
\end{definition}

The Cauchy-Schwarz quotient is bounded from above by one
which attained exactly at the eigenvectors, except for the eigenvalues $0$ and $\infty$. 
(Then the Cauchy-Schwarz quotient is not defined.)
So modulo these exceptions, the spectrum is non-empty if and only if one is attained.
For homogeneous problems $z$ can be restricted to be of unit length. In particular, for the significance of this tool in the numerical solution of large linear eigenvalue problems, see \cite[Section 3.2]{HUKO}.
% 
%\smallskip
%
%\begin{example} 

Suppose one is not attained. 
Then the deviation of \eqref{reali} from one, when 
$z\in \setC^n$ varies among unit vectors, 
gives a measure of emptiness of the spectrum away from $0$ and $\infty$. 

\smallskip

\begin{example} Continuing Example \ref{lineq},
consider the linear system \eqref{solu}.
Then $M$ is invertible if and only if the Cauchy-Schwarz quotient of the eigenvalue problem \eqref{cohen} attains one.
\end{example}

\smallskip

For an illustration of an upper bound, consider Example \ref{paha}. Then the Cauchy-Schwarz quotient reads
$$\frac{|(Mz,z)|^2}{\|Mz+T\overline{z}\|^2}.$$
If $M=0$ and $T$ is invertible, then this is identically zero, giving a "maximal degree of emptiness" of the spectrum.
For a more subtle estimate, see \cite[Corollary 3.12]{HUPFA}.

\smallskip

\begin{example} 
In the linear case \eqref{gen} the 
Cauchy-Schwarz 
quotient \eqref{reali} reads
$$\frac{|(Mz,Nz)|^2}{\|Mz\|^2\|Nz\|^2}$$ of which taking the conjugate co-gradient gives
rise 
to the homogeneous nonlinear eigenvalue problem
$$(Mz,Nz)M^*Nz+(Nz,Mz)N^*Mz=\lambda 
\left((Mz,Mz)N^*Nz+(Nz,Nz)M^*Mz\right);$$
see \cite[Eq. (5.1)]{HUKO} for taking the conjugate co-gradient.
One can verify that any eigenvector of \eqref{gen} 
is also an eigenvector of this problem corresponding to the eigenvalue one. (Exception: eigenvectors 
$z$ corresponding to $0$ or $\infty$ give $0=\lambda 0$.) 
In particular, any linear eigenvalue problem can be solved by solving a nonlinear eigenvalue problem. If the problem is standard, then 
this equation reads
\begin{equation}\label{slor}
({\rm rq}_{M,I}(z)M^*+\overline{{\rm rq}_{M,I}(z)}M-M^*M)z=
{\rm rq}_{M^*M,I}(z)z
\end{equation}
once we set $\lambda=1$. This can be regarded as being a mildly non-linear equation involving Hermitian matrices.
\end{example}

\smallskip

The eigenvalues are bundled at the Cauchy-Schwarz quotient
attaining one. Consequently,  ascending by using the corresponding conjugate co-gradient requires a good initial guess for an eigenvector of interest.

%\begin{theorem} Suppose $A(z,\overline{z})=Mz+f(z)$, where
%$$f(z,\overline{z})=
%f_1(|z|)\circ z+f_2(|z|)\circ\frac{1}{\overline{z}}$$
%with a Hermitian matrix $M\in \setC^{n \times n}$ and $f_1$ %and $f_2$  with real coefficients and $B(z,\overline{z})=z$. %Then the spectrum of \eqref{swes} 
%real.
%\end{theorem}
%
%\begin{proof} The claim follows immediately from the %observation that the inner-product of 
%$A(z,\overline{z})$ and $B(z,\overline{z})$ is always real.
%\end{proof}
%
%\smallskip
%
%Denote by $L(z)$ the linearization of $A(z)$ at $0$. 
%This means
%partial derivation with respect to $z$ and $\overline{z}$
%and that $L(z)$ is, in general, a real linear operator in $%\setC^n$. Because the scaling in the previous example was real, the limiting behaviour goes analogously with any eigenvalue of $L$. This gives rise to the following notion.

%\begin{definition} In the eigenvalue problem \eqref{swes}, denote by $L(z)$ the linear
%part of $A(z)$ at $0$ and let $B(z)=z$. Then $\sigma(L)$ is 
%said to be the linear spectrum of $A(z)$.
%\end{definition}
%
%Obviously the linear part at $0$ can vanish. If $L\not=0$ 
%and genuinely real linear, then $\sigma(L)$ can be large and challenging to compute when $n$ large; see
%\cite{HUPFA}. It is also possible that $\sigma(L)=\emptyset$.

%\smallskip
%
%\begin{example} In the nonlinear Rayleigh quotient problem
%for the $p$-Laplacian \cite{linkku} one deals, in the matrix formulation, with
%$$M(|Mz|^{p-2}\circ Mz)=\lambda |z|^{p-2}\circ z$$
%with a matrix $M \in \setC^{n \times n}$.
%With $p\not=2$ this is a nonlinear eigenvalue problem. 
%\end{example}
%
\smallskip

Assume now having $z\in \setC^n$ possibly resulting from
ascending a few steps by using conjugate co-gradients of
the Cauchy-Schwarz quotients. If the Cauchy-Schwarz quotient is sufficiently close to one, the task is to produce eigenvalue estimates. As argued earlier, \eqref{skers} can be  useless in this regard. For an arguable approach we
follow \cite{HUKO}, relying entirely on measuring linear independence. This means choosing a single direction  
$w\in \setC^n$ where $A(z,\overline{z})$ and 
$B(z,\overline{z})$ in \eqref{swes} appear to be simultaneously pointing at.
Taking the inner-product both sizes of \eqref{skers} then determines the quotient. We suggest choosing the unit vector $w$
to satisfy
$$\max_{\|w\|=1}
\left( |w^*z_1|^2+|w^*z_2|^2
\right),
$$
where $z_1=\frac{A(z,\overline{z})}{\|
A(z,\overline{z})\|}$ and 
$z_1=\frac{B(z,\overline{z})}{\|
B(z,\overline{z})\|}$. This gives, analogously to \eqref{bis1ar}, the quotient
\begin{equation}\label{bbis1ar}
    {\rm oq}_{A,B}(z)= \frac{(A(z,\overline{z}),B(z,\overline{z}))}{|(A(z,\overline{z}),B(z,\overline{z}))|}
    \frac{\|A(z,\overline{z})
    \|}{\|B(z,\overline{z})\|}
  \end{equation}
approximating the eigenvalue.
\begin{theorem} Assume $z=w+\epsilon\in \setC^n$ where
$w$ is an eigenvector of \eqref{swes} corresponding to an
eigenvalue $\lambda\not=0,\infty$.
Then ${\rm oq}_{A,B}(z)=\lambda +O(\epsilon)$.
\end{theorem}

\begin{proof} This follows by the continuity assumption, i.e., by using $A(w+\epsilon,\overline{w +\epsilon})=
A(w,\overline{w})+O(\epsilon)$ and
$B(w+\epsilon,\overline{w +\epsilon})=
B(w,\overline{w})+O(\epsilon)$.
 \end{proof}

%When the spectrum is a subset of $\setR$, let us consider
%\begin{equation}
%\min_{\mu \in \setR}
%\{(A(z),B(z)) -\mu (B(z),B(z))\geq 0\;|\; \mbox{ for all } \, z\}
%\end{equation}
%giving a lower bound on the smallest eigenvalue.

%\begin{theorem} For the Gross-Pitaevskii eigenvalue problem %\eqref{grop} the smallest eigenvalue of
%$M$ bounds the eigenvalues from below.
%\end{theorem} 

%\begin{proof} Choose $\mu$ to be the smallest eigenvalue of 
%$M$. Then $M-\mu I$ is positive semidefinite and, 
%since $\beta$ is nonnegative, the inner-products of $(A-\mu %I)(z)$ and $B(z)$ are non-negative real for any $z$. 
%\end{proof}

\section*{Conclusions}
Computing conjugate co-gradients of real-valued quotients 
gives rise to eigenvalue problems, both linear and nonlinear. In the linear Hermitian case conjugate co-gradients of optimal quotients yield a notion of generalized folded spectrum eigenvalue problem.
Replacing the Euclidean norm in optimal quotients with the 
$p$-norm, a matrix version of the so-called $p$-Laplacian eigenvalue problem arises.
Such problems seem to be naturally classified as being a special case of homogeneous eigenvalue problems. %particular, a matrix version of the $p$-Laplacian eigenvalue problem.
%This particular eigenvalue problem has many interpretations.
Being a quite general class, tools are developed for recovering whether a given homogeneous eigenvalue problem is a gradient eigenvalue problem. 
It turns out to be a delicate issue to come up with a valid quotient.
A notion of nonlinear Hermitian eigenvalue problem arises.
Cauchy-Schwarz quotients are introduced as an option to deal with tough problems.

\end{document}